\documentclass[12pt]{article}
\usepackage{verbatim}
\usepackage[margin=1.5in,vmargin=1.5in]{geometry}
\usepackage[utf8x]{inputenc}
\usepackage[T1]{fontenc}
\usepackage{lmodern}
\usepackage[english]{babel}
\usepackage[draft=false,kerning=true]{microtype}
\usepackage{ifthen}
\usepackage{amsmath, amsthm, amssymb}
\usepackage{graphicx}
\usepackage{tikz}
\usepackage{fancyhdr}
\usepackage{hyperref}

\numberwithin{equation}{section}

\newcommand{\dd}{\mathrm{d}}

\newcommand\matA{{\bf A}}
\newcommand\matM{{\bf M}}

\newcommand\vecu{{\bf u}}
\newcommand\vecX{{\bf X}}
\DeclareMathOperator\Exp{Exp}

\DeclareMathOperator\Gam{Gamma}
\DeclareMathOperator\Bin{Bin}

\newcommand\Biggiven{\, \Big \vert \, }
\newcommand\convD{{\buildrel {\mathcal D} \over \longrightarrow}}

\newcommand{\R}{\mathbb{R}}
\newcommand{\Z}{\mathbb{Z}}

\newcommand{\parentheses}[4][]%
{\mathopen{}\ifthenelse{\equal{#1}{}}{\left#2}{\csname#1\endcsname#2}%
    {#4}\mathclose{}\ifthenelse{\equal{#1}{}}{\right#3}{\csname#1\endcsname#3}}
\newcommand{\foperator}[1]{\ensuremath{%
    \mathop{{#1}\thinspace\negthinspace}
    \mathchoice{\negthinspace}{\negthinspace}{}{}}}
\newcommand{\f}[3][]{\ensuremath{\foperator{#2}\parentheses[#1]{(}{)}{#3}}}

\newcommand{\Oh}[2][]{\ensuremath{O\parentheses[#1]{(}{)}{#2}}}

\DeclareMathOperator\E{\mathbb{E}}
\DeclareMathOperator\Prob{\mathbb{P}}
\DeclareMathOperator\V{\mathbb{V}ar}
\DeclareMathOperator\Cov{\mathbb{C}ov}
\newcommand{\fE}[2][]{\E\parentheses[#1]{[}{]}{#2}}
\newcommand{\fP}[2][]{\Prob\parentheses[#1]{(}{)}{#2}}
\newcommand{\fV}[2][]{\V\parentheses[#1]{(}{)}{#2}}
\newcommand{\fCov}[2][]{\Cov\parentheses[#1]{(}{)}{#2}}
\newcommand{\fexp}[2][]{\exp\parentheses[#1]{(}{)}{#2}}

\theoremstyle{plain}
\newtheorem{theorem}{Theorem}[section]
\newtheorem{lemma}[theorem]{Lemma}

\theoremstyle{remark}
\newtheorem{remark}[theorem]{Remark}

\newcommand\polya{P\'olya}

\newcommand{\TODO}[1]%
{\par\fbox{\begin{minipage}{0.9\linewidth}\textbf{TODO:} #1\end{minipage}}\par}

\begin{document}

\title{The continuum \polya-like random walk}
\fancyhead[CO]{\footnotesize The continuum \polya-like random walk}
\author{Daniel Krenn\thanks{Daniel~Krenn is supported by the
    Austrian Science Fund (FWF): P\,24644-N26.},
  Hosam\@ Mahmoud, and Mark Daniel Ward\thanks{Mark Daniel Ward is
    supported by NSF Grant DMS-1246818, and by the NSF Science \&
    Technology Center for Science of Information Grant CCF-0939370.}}
\fancyhead[CE]{\footnotesize Daniel Krenn, Hosam\@ Mahmoud, and Mark Daniel Ward}
\fancyhead[LE, RO]{\footnotesize \thepage}
\date{}

\maketitle

\begin{abstract}
The \polya\ urn scheme is a discrete-time process 
concerning the addition and removal of colored balls. 
There is a known embedding of it in continuous-time called the
\polya\ process.
We deal with a generalization of this stochastic model,
where the initial values and the entries of the transition
matrix (corresponding to additions or removals) are not necessarily
fixed integer values as in the standard \polya\ process. In one of
the scenarios, we even allow the entries of the matrix to be random
variables.
As a result, we no longer have a combinatorial model of ``balls in an
urn,'' but a broader interpretation as a random walk
in a possibly high number of dimensions.
In this paper, we study several parametric classes of
these generalized continuum \polya-like random walks.
\end{abstract}
 
\bigskip
{\footnotesize

}

\bigskip
\section{Introduction}
\label{Sec:intro}

We
deal with a generalization of the \polya\ process on $c$ colors
by embedding each element of the replacement
matrix into~$\R$, or by even treating the elements of the replacement
matrix as random variables.
We are no longer restricted to thinking about integer-valued
quantities (such as \emph{counts} of 
balls in an urn),
we can now handle real-valued quantities, such as a random walk with
possibly fractional step sizes, in $c$ dimensions, rather than
restricting to integer step sizes in each of the $c$ dimensions.
We can view the
replacement matrix itself as random.
The generalized model 
is mentioned
in~\cite{Janson:2004:functional-limit-thms-branching,
  Pouyanne:2008:algebraic-approach-polya-processes}.

\subsection{Background}
We first review the standard \polya\ process.  The \polya\ urn scheme is a 
process underlying an urn
that evolves in discrete
time. The urn contains balls of up to $c$ colors. 
The colors are numbered, say the set of colors is
$C=\{1,\dots,c\}$. 
At each discrete epoch in time, a ball is sampled from
the urn.  It is then put back in the urn, together with
a number of other balls in various colors. If the sampled ball
has color $i\in C$, then we add to the urn
$A_{i,j}$ balls of color $j$, for $j\in C$. 
If $A_{i,j}$ is negative,
we remove $A_{i,j}$ balls of color $j$.
It is customary to represent these dynamics by a replacement matrix
$$\matA = \begin{pmatrix} A_{1,1} &A_{1,2} &\ldots &A_{1,c}\cr
                   A_{2,1} &A_{2,2}& \ldots &A_{2,c}\cr
                   \vdots &\vdots &\ddots &\vdots\cr
                   A_{c,1} &A_{c,2} &\ldots &A_{c,c}\end{pmatrix}.$$
In the standard \polya\ process each element of $\matA$ is an integer.
(We will deal with an extended view that considers $\matA$ as having 
real-valued entries in 
Section~\ref{generalizedpolyaprocess}.)
It is usually assumed that the urn is ``tenable''
in the sense that the selection of balls can be continued ad
infinitum, no matter which stochastic path is followed, i.e., the
process will never get stuck.

The \polya\ process is an embedding of the \polya\ urn scheme in real
time. Embedding in real time (poissonization) was suggested by
Kac~\cite{Kac:1949:deviations} as a general methodology for
understanding discrete probability problems. In the context of urns,
poissonization was introduced by Athreya and Karlin~\cite{Athreya-Karlin:1968:urn-continuous} to understand discrete-time 
\polya\ urn schemes. 
Poissonization was thus meant as a transform.  
The inverse transform (depoissonization)---to translate results in the continuous domain
back to results in the discrete domain---is fraught with difficulty~\cite{Athreya-Karlin:1968:urn-continuous}. 
For general background on depoissonization,
we refer the reader to~\cite{Hofri-Mahmoud:2019:algo-nonuniformity} (Chapter~8),
or~\cite{Szpankowski:2001:average-acse-alg-sequences} (Chapter~10). 

Some authors developed interest in the continuous-time \polya\ process for its own 
sake (see \cite{Balaju-Mahmoud:2006:limit-diag,
  Balaju-Mahmoud-Watanabe:2006:ehrenfest,
  Chen-Mahmoud:2016:time-continuous-polya,
  Sparks-Mahmoud:2013:two-color-polya}). 
In the \polya\ process, each ball carries an internal clock that rings
in $\Exp(1)$ time (a random amount of time, according to an exponential random variable with mean~1), independently of the behavior of all other clocks. Whenever a clock rings, it is instantaneously reset to ring again in  
$\Exp(1)$ time (independently of the clocks on all the
other balls). 
In other words, each ball has the ability to generate a new Poisson
process with intensity~1.  
When the clock associated with a ball of color $i$ rings,
the addition and
removal of balls corresponds to picking a colored ball from a
\polya\ urn and using the $i$th row of the replacement matrix $\matA$
to determine which balls to add or remove. 
All replacements are assumed to occur 
instantaneously, and each new ball is given an independent clock that
rings in $\Exp(1)$ time.
\subsection{Generalized \polya\ process}\label{generalizedpolyaprocess}
In this paper, we study a generalized \polya\ process that can be viewed as a second layer of embedding of the \polya\ urn scheme.
The starting numbers and the (possibly random) numbers added are 
in~$\R$ (no longer restricted to~$\Z$),
so the numbers are no longer counts of balls. A natural (more
general) interpretation is a random walk in $c$ dimensions.  At time
$t$, the current position of the walk in~$\R^c$ is a $c$-dimensional column vector
$\vecX (t) := (X_1(t),\ldots,X_c(t))^\top$.

At any point $t$ in time, the next renewal occurs after 
a random amount of time, according to a master clock.
The waiting
time of this master clock follows an exponential random variable
with mean~$1/\sum_{i\in C} X_i(t)$.
Given that a renewal occurs, the probability that the renewal
corresponds to color $i$ is $X_{i}(t)/\sum_{j\in C}X_{j}(t)$.
In other words, the probability of a renewal of type~$i$ is
proportional to the amount of quantity~$i$ present when the
transition occurs.  In such a case, the $i$th row of the matrix
$\matA$ dictates the $c$-dimensional direction in which to move, i.e.,
we add $A_{i,j}$ units to $X_j(t)$, for each $j\in C$.
Thus, it is appropriate then to call $\matA$ 
the \emph{navigation matrix}, instead of the replacement matrix. 

To avoid trivialities, we only consider starting values~$X_j(0)$
and matrices~$A_{i,j}$ in which the walk is tenable, i.e., the walk
always avoids the origin, and each of the coordinates is always
nonnegative.  We call the row vector $\vecX^\top (t) =
(X_1(t),\ldots,X_c(t))$ a continuum \polya-like process or random walk.

This stochastic process is not a Poisson process,
because the rate of the process itself is random, i.e., the rate of
replacement is not simply a function of time.  The current rate of
replacement depends on the number of replacements that have taken
place beforehand.

\subsection{Organization of the paper}

The paper is organized as follows. 
Section~\ref{sec:model} specifies the probability model for
the stochastic process.  In Section~\ref{Sec:pde},
we derive a fundamental partial differential equation that governs
the behavior of the generalized \polya\ process. 
In Section~\ref{Sec:ave}, we derive a functional equation for
the moment generating function of the position of the random walk at
time $t$. Section~\ref{sec:balanced-triangular} is perhaps the most 
novel part of the paper, because we solve the partial differential
equations for a balanced upper-triangular case, 
a case that proved difficult in the usual urn setting.

Additionally, in Appendix~\ref{sec:examples}, we show how 
several classical probability models can be generalized with this
approach.  In all of these classical cases, the 
partial differential equations from Section~\ref{Sec:ave} can be
solved.

\section{The probabilistic model}
\label{sec:model}

In order to be able to establish the partial differential equations
of Section~\ref{Sec:pde}, we need to precisely describe 
the number of renewals, say $N(t,\Delta t)$, that occur
in the processes during the interval $(t, t+\Delta t]$.
We use the notation ${\mathcal C}_i$ to indicate the event that exactly one renewal occurs 
in the interval $(t, t+\Delta t]$ and that renewal is induced by 
 the $i$th coordinate.
To study the behavior of $\vecX (t+s)$, for $s > 0$, we condition on
the vector $\vecX (t)$.
\begin{lemma}\label{lem:P:events}
The conditional probabilities of either zero renewals, one renewal of 
type~$i$, or two or more renewals, in the interval $(t, t+\Delta t]$,
given the value of~$\vecX(t)$, are (respectively) the following:
  \begin{align*}
    \fP[big]{N(t,\Delta t) = 0\ |\  \vecX (t)} &=
    \fexp[Big]{-\Delta t \sum_{j\in C}  X_j(t)},\\
    \fP[big]{\text{$(N(t,\Delta t) = 1) \cap {\mathcal C}_i$} \ |\  \vecX (t)} &=
    \Delta t\, X_i(t) \fexp[Big]{-\Delta t \sum_{j\in C}  X_j(t)} +
    \Oh{(\Delta t)^{2}},\\
    \fP[big]{N(t,\Delta t) \geq 2\ |\  \vecX (t)} &= 
    \Oh{(\Delta t)^{2}},
  \end{align*}
  as $\Delta t \to 0$.
\end{lemma}

\begin{proof}
Given the values $ \vecX (t)$, the master clock does not
ring during the interval $(t,t+\Delta t]$ with probability
\begin{equation*}
\fP[big]{N(t,\Delta t) = 0\ |\  \vecX (t)} =
\prod_{j\in C} \biggl(\frac {(\Delta t)^0 e^{-\Delta t}} {0!}\biggr)^{X_j(t)}
= \fexp[Big]{-\Delta t \sum_{j\in C}  X_j(t)}.
\end{equation*}

For fixed $t$, once we are given the values $ \vecX (t)$, the next ring of the
master clock, after time $t$, occurs at a random time $x$ (with $t < x$) with probability density function
$$\Big(\sum_{j\in C}  X_j(t)\Big)\fexp[Big]{-(x-t) \sum_{j\in C}  X_j(t)}.$$
When such a clock ring occurs at time $x$, it is a ring of type~$i$
with probability $X_{i}(t)/\sum_{j\in C}X_{j}(t)$.
Then, for $t < x \leq t + \Delta t$, there are no additional subsequent
rings before time $t + \Delta t$ with probability
\begin{multline*}
\prod_{j\in C} \biggl(\frac {(t+\Delta t-x)^0 e^{-(t+\Delta t-x)}} {0!}\biggr)^{X_j(t)+A_{i,j}} \\
= \fexp[Big]{-(t+\Delta t-x) \sum_{j\in C}  \bigl(X_j(t) + A_{i,j}\bigr)}.
\end{multline*}

Putting all of this together
(and again given $\vecX(t)$),
the conditional probability 
$\fP[big]{(N(t,\Delta t) = 1) \cap \mathcal C_i\ |\  \vecX (t)}$
that the master clock rings exactly once 
during the interval $(t,t+\Delta t]$
and yields a renewal of type~$i$,
is equal to
\begin{align*}
\int_{t}^{t+\Delta t}&
\Big(\sum_{j\in C}  X_j(t)\Big)\fexp[Big]{-(x-t) \sum_{j\in C}  X_j(t)}
\bigg(\frac{X_{i}(t)}{\sum_{j\in C}X_{j}(t)}\bigg) \\
&\times\fexp[Big]{-(t+\Delta t-x) \sum_{j\in C}  \bigl(X_j(t) + A_{i,j}\bigr)}
\, \dd x.
\end{align*}
This simplifies to
\begin{align*}
&\int_{t}^{t+\Delta t}
X_{i}(t)\,\fexp[Big]{
-\Delta t\sum_{j\in C}  X_j(t)-(t+\Delta t-x) \sum_{j\in C}A_{i,j}}\!\ \dd x\\
&\qquad{}= \frac{ X_{i}(t)\,\fexp[Big]{-\Delta t\sum_{j\in C} X_j(t)}\Bigl(
\Delta t\sum_{j\in C}A_{i,j}
+ \Oh{(\Delta t)^{2}}
\Bigr) }{\sum_{j\in C}A_{i,j}}\\
&\qquad{}= \Delta t\,X_{i}(t)\fexp[Big]{-\Delta t\sum_{j\in C} X_j(t)}
+ \Oh[big]{(\Delta t)^{2}}.
\end{align*}
Finally, the conditional probability (given $\vecX(t)$) that the master clock rings two
or more times during the interval $(t,t+\Delta t]$
is 
\begin{align*}
\fP{N(t,\Delta t) \geq 2\ |\  \vecX (t)} 
&= 1 
    - \fP{N(t,\Delta t) = 0\ |\  \vecX (t)} \\
       &\hspace*{1.85em}{} - \sum_{i\in C}\fP{(N(t,\Delta t) = 1) \cap {\mathcal C}_i\ |\  \vecX (t)}\\
&= 1 - \fexp[Big]{-\Delta t \sum_{j\in C}  X_j(t)}\\
        &\hspace*{1.85em}{} - \sum_{i \in C}\Delta t\, X_i(t) \fexp[Big]{-\Delta t \sum_{j\in C}  X_j(t)}
+ \Oh[big]{(\Delta t)^{2}}\\
&= \Oh[big]{(\Delta t)^{2}} ;
\end{align*}
 we arrive at the latter conclusion after a local expansion of the two exponential
 functions.
This completes the proof of Lemma~\ref{lem:P:events}.
\end{proof}

\section{The fundamental partial differential equation}
\label{Sec:pde}

We formulate here a partial differential equation for the continuum
\polya-like random walk. We use the vector $\vecu = (u_j)_{j\in C}$ to mark
the colors $C=\{1,\dots,c\}$.  Let
\begin{equation*}
  \phi(t, \vecu) = \fE[Big]{\fexp[Big]{\sum_{j\in C} u_j X_j(t)}}
\end{equation*}
be the joint moment generating function
of the coordinates of the random walk~$\vecX(t)$.  For $i\in C$, let
\begin{equation*}
  \psi_i(\vecu) = \fE[Big]{\fexp[Big]{\sum_{j\in C} u_j A_{i, j}}}
\end{equation*}
be the joint moment generating function of the random variables
on row $i$ of the navigation matrix~$\matA$.

\begin{theorem}
  \label{Theo:pde}
  The joint moment generating function $\phi(t, \vecu)$ satisfies
\begin{equation*}
  \frac{\partial \phi}{\partial t}  + \sum_{i\in C} (1 - \psi_i)
  \frac{\partial \phi}{\partial u_i} = 0.
\end{equation*}
\end{theorem}

\begin{proof}
  We use conditional expectation to calculate
  \begin{equation*}
    E := \fE[Big]{\fexp[Big]{\sum_{j\in C} u_j X_j(t+\Delta t)} \Biggiven \vecX (t)},
  \end{equation*}
  i.e., the expectation conditioned on the status~$\vecX (t)$
  at time~$t$. We do this by first conditioning on whether there are
  zero, one or
  at least two renewals in the interval $(t, t+\Delta t]$.
  If there is exactly
  one renewal, we also condition on the color of the chosen %
  direction.
  The probabilities of these events are calculated in
  Lemma~\ref{lem:P:events}. As in Section~\ref{sec:model}, let
  $N(t,\Delta t)$ denote the number of renewals that occur in the
  processes in the interval $(t, t+\Delta t]$.

  We obtain
  \begin{align*}
    E
    &= \fE[Big]{\fexp[Big]{\sum_{j\in C} u_j X_j(t+\Delta t)}
      \Biggiven \vecX (t) \text{ and } N(t,\Delta t)  = 0} \times \fP[big]{N(t,\Delta t) = 0}\\
    &\phantom{=}\hphantom{0} +
    \sum_{i\in C}
   \fE[Big]{\fexp[Big]{\sum_{j\in C} u_j X_j(t+\Delta t)}
      \Biggiven \vecX (t) \text{ and } (N(t,\Delta t) = 1) \cap \mathcal C_i} \\
         &\qquad \qquad {} \times
    \fP[big]{N(t,\Delta t) = 1,\, \mathcal C_i}\\
     &\phantom{=}\hphantom{0}+
     \fE[Big]{\fexp[Big]{\sum_{j\in C} u_j X_j(t+\Delta t)}
      \Biggiven \vecX (t) \text{ and } N(t,\Delta t)  \ge 2} \times \fP[big]{N(t,\Delta t) \ge 2}.
   \end{align*}
  In the last equality, we have allowed for the possibility that the
  $A_{i,j}$ themselves (i.e., the entries of the navigation matrix)
  may be random variables, allowing for more generality.  
  We only assume that these matrix entries are
  independent of the current state of the process $\vecX(t)$.

  Collecting all these facts
  and utilizing Lemma~\ref{lem:P:events}, we obtain
  \begin{align*}
    E &= \fexp[Big]{\sum_{j\in C} u_j X_j(t)} 
    \fexp[Big]{-\Delta t \sum_{i\in C}  X_i(t)}\\
    &\phantom{=}\hphantom{0} +
    \sum_{i\in C} 
    \fE[Big]{\fexp[Big]{ \sum_{j\in C} u_jA_{i,j}} }
    \fexp[Big]{\sum_{j\in C} u_j X_j(t)}\Delta t\, X_i(t)\\
    &\qquad \qquad {} \times \fexp[Big]{-\Delta t \sum_{j\in C}  X_j(t)} +
    \Oh{(\Delta t)^{2}}.
  \end{align*}
  A local expansion of the exponentials gives
  \begin{align*}
    E &= \fexp[Big]{\sum_{j\in C} u_j X_j(t)}
    \Bigl(1 -\Delta t \sum_{i\in C}  X_i(t) + \Oh{(\Delta t)^2} \Bigr)\\
    &\phantom{=}\hphantom{0} +
    \Delta t \bigl(1 + \Oh{\Delta t}\bigr)
    \fexp[Big]{\sum_{j\in C} u_j X_j(t)} 
    \sum_{i\in C} X_i(t)
    \fE[Big]{\fexp[Big]{ \sum_{j\in C} u_jA_{i,j}} }\\
    &\phantom{=}\hphantom{0} +
    \Oh{(\Delta t)^{2}}.
  \end{align*}
 
  Taking expectations over $\vecX(t)$ yields
  \begin{align*}
    \phi(t + \Delta t, \vecu)
    = \phi(t, \vecu)
    &- \Delta t \sum_{i\in C} \fE[Big]{X_i(t)
    \fexp[Big]{\sum_{j\in C} u_j X_j(t)}}\\
    &{}+
    \Delta t \bigl(1 + \Oh{\Delta t}\bigr)
    \fE[Big]{\fexp[Big]{\sum_{j\in C} u_j X_j(t)} 
    \sum_{i\in C} X_i(t)} \psi_i(\vecu)\\
    &{}+
    \Oh{(\Delta t)^{2}}.
\end{align*}
We can now write the limiting form
\begin{align*}
  \frac{\partial \phi(t, \vecu)}{\partial t}
  &= \lim_{\Delta t \to 0} \frac{\phi(t + \Delta t, \vecu) 
    - \phi(t, \vecu)} {\Delta t} \\
  &= - \sum_{i\in C} \fE[Big]{  X_i(t) 
  \fexp[Big]{\sum_{j\in C} u_j X_j(t)} }
  \bigl( 1 - \psi_i(\vecu)\bigr) \\
  &= - \sum_{i\in C} \frac {\partial \phi(t, \vecu)} {\partial u_i}
 \bigl (1- \psi_i(\vecu)\bigr).
\end{align*}
\end{proof}
\begin{remark}
  The proof of Theorem~\ref{Theo:pde} is a
  generalization of the proof in Balaji and Mahmoud~\cite{Balaju-Mahmoud:2006:limit-diag}. However,
  it needed some new techniques. In~\cite{Balaju-Mahmoud:2006:limit-diag}, there is a
  conditional argument that uses a sum on the number of balls of a
  color, given that number. Of course, the number of balls is an
  integer and such a sum can be carried out. 
  Here, the counterpart of
  a number of balls of a certain color is a continuous coordinate, 
  and such a conditional sum cannot be written.
\end{remark}
\section{Moments}
\label{Sec:ave}
We derive functional equations for the moments.  For example, we
derive a functional
equation for the mean position (see Theorem~\ref{Theo:polyaprocave} at
the end of this section) by differentiating 
(with respect to~$u_j$) on both sides of the partial differential
equation in Theorem~\ref{Theo:pde}.  This yields
\begin{equation*}
  \frac{\partial}{\partial u_j}
  \mathopen{}\left(\frac{\partial \phi}{\partial t}\right)
  + \sum_{i\in C} \bigl(1 - \psi_i\bigr)
  \frac{\partial^2 \phi}{\partial u_j\, \partial u_i}
  - \sum_{i\in C} \frac{\partial \psi_i}{\partial u_j} \times
  \frac{\partial \phi}{\partial u_i} = 0.
\end{equation*}
Evaluation of the summands at $u_i=0$, for all $i\in C$, yields
\begin{align*}
   \frac{\partial}{\partial u_j}
  \mathopen{}\left(\frac{\partial \phi}{\partial t}\right)  \Bigg|_{\vecu=0}
  &=\frac{\partial}{\partial t}
  \mathopen{}\left(\frac{\partial \phi}{\partial u_j}\right)
  \Bigg|_{\vecu=0} \\
  &= \frac{\partial}{\partial t}
  \fE[Big]{X_j(t) \fexp[Big]{\sum_{i\in C} u_i X_i(t)}} \Bigg|_{\vecu=0}\\
  &= \frac{\dd}{\dd t} \fE{X_j(t)}.      
\end{align*}
We also have
\begin{equation*}
  \sum_{i\in C} \bigl(1 - \psi_i\bigr)
  \frac{\partial^2 \phi}{\partial u_j\, \partial u_i}\Big| _{\vecu=0}= 0,
\end{equation*}
because $(1-\psi_i)\big| _{\vecu=0}= 0$,
and
\begin{equation*}
  \sum_{i\in C}
  \frac{\partial \psi_i}{\partial u_j} \times
  \frac{\partial \phi}{\partial u_i}  \Big| _{\vecu=0}
  = \sum_{i\in C} \fE{A_{i,j}} \fE{X_i(t)}.
\end{equation*}
Hence, for the $j$th coordinate, we get the ordinary differential equation  
\begin{equation*}
  \frac{\dd}{\dd t} \fE{X_j(t)} =
  \sum_{i\in C}  \fE{A_{i,j}} \fE{X_i(t)}.
\end{equation*}
Putting the differential equations (for $j\in C$)
together in matrix form, we get the functional equation
\begin{equation*}
  \frac{\dd}{\dd t} \fE[big]{\vecX(t)} = \fE[big]{\matA^\top} \fE[big]{\vecX (t)},
\end{equation*}
where $\matA^\top$ is the transpose of the navigation
matrix~$\matA$. This first-order functional equation has a standard
solution; this yields the following theorem.

\begin{theorem}
\label{Theo:polyaprocave}
Let $\matA$ be the navigation matrix of a continuum \polya-like random walk.
At time $t$, the expected value of the coordinates of the walk are
\begin{equation*}
  \fE{\vecX(t)} = e^{\fE[empty]{\matA^\top}\, t}\,  \vecX(0).
\end{equation*}
\end{theorem}
Note that, for a matrix $\matM$,
we have used the notation
$e^{\matM} = \sum_{n=0}^\infty {\matM^n}\!/ n!$,
and we can compute this by using the Jordan form of $\matM$.
\begin{remark}
  The result of Theorem~\ref{Theo:polyaprocave} can be obtained via an
  alternative martingale approach. It is discussed
  in~\cite{Janson:2004:functional-limit-thms-branching}
  that $e^{-t \E[A^\top]}\vecX(t)$ is martingale. It follows that
  $\E[e^{-t \E[A^\top]}\vecX(t)] = \vecX(0)$, or in other words,
  $\E[\vecX(t)] = e^{t \E[A^\top]}\vecX(0)$.  However, it is harder to
  obtain higher moments by this technique. On the other hand,
  martingale convergence theorems can give us (nonconstructively)
  almost sure limits, without specifying what the limits are. The PDE
  method being discussed here works hand in hand with that alternative
  line to give us the distributions.
\end{remark}

\def\moveexamples{
\subsection{\polya--Eggenberger-like random walk}
\label{sec:polya-eggenberger}

Suppose the navigation
matrix~$\matA$ is in the diagonal matrix form
\begin{equation*}
  \begin{pmatrix} 
    A_1&0&  \cdots &0\\
    0 &A_2& \cdots &0\\
    \vdots &\vdots & \ddots &\vdots\\
    0 &0 & \cdots &A_c
  \end{pmatrix},
\end{equation*}
where $A_i$ is a nonnegative random variable. Let us focus on the
$i$th component; we set $u=u_i$ and $u_j=0$ for $j\neq i$.
Let $\psi_i(u) = \fE[empty]{e^{u A_i}}$ be the moment generating function of $A_i$, and
let $\phi_i(t,u) = \fE[empty]{e^{u X_i(t)}}$ be the moment generating
function of $X_i(t)$.
The partial differential equation simplifies to
\begin{equation}\label{eq:pde:polya-egg}
  \frac{\partial \phi_i(t, u)}{\partial t}  + 
  (1 - \psi_i (u) ) \frac{\partial \phi_i(t, u)}{\partial u}
  = 0.
\end{equation}
This equation can be solved for several standard distributions of the
$A_i$. Two examples are discussed below.

\subsubsection{Component-wise almost surely constant}

Suppose $A_i = \alpha_i \in \R^+$ almost surely. This
gives rise to
\begin{equation*}
  \phi_i(t, u) =
  (1 - e^{\alpha_i t} (1-e^{-\alpha_i u}))^{-X_i(0)/\alpha_i}.
\end{equation*}
We set $u = s e^{-\alpha_i t}$ and take the limit, 
to obtain
\begin{equation*}
\lim_{t\to\infty}  \fE[Big]{\exp \Bigl(\frac {sX_i(t)} {e^{\alpha_i t}}\Bigr)}= \lim_{t\to\infty} \phi_i\Bigl(t, \frac s {e^{\alpha_i t}}\Bigr) 
       = (1 - \alpha_i s)^{-X_i(0)/\alpha_i}.
     \end{equation*}
The latter moment generating function is that of a $\Gam(X_i(0)/\alpha_i, \alpha_i)$ random
variable. That is, we have
\begin{equation*}
\frac {X_i(t)} {e^{\alpha_i t}} \ \convD \ \Gam\Bigr(\frac {X_i(0)} {\alpha_i}, \alpha_i\Bigl).
\end{equation*}
Note that the displacements along the $i$th coordinate affect only changes in that direction.
In other words, the limit multivariate distribution has independent marginals, with the $i$th marginal having the latter gamma limit distribution.
Also observe that this random walk has a very long memory. It never forgets where it starts.
Even the limit is influenced by the initial position vector, which comes in as a parameter
in the joint limit distribution.  

\subsubsection{Component-wise exponential distribution}

We illustrate with another instance, \emph{in which the navigation
  matrix itself has random elements}. Suppose 
the \polya--Eggenberger-like random walk operates under exponentially 
distributed displacements. That is, the $A_i$ (for $i\in C$) are
independent $\Exp(1)$ random variables.
Thus, for $u < 1$, we have $\psi_i(u)=1/(1-u)$ and we can
solve~\eqref{eq:pde:polya-egg}. 
In this case, we have
\begin{equation*}
\phi_i(t, u) = e^{-\f[empty]{W}{-u e^{t-u}}},
\end{equation*}
where $\f{W}{\,\cdot\,}$ is Lambert's $W$~function (defined implicitly as any complex solution of
$z = \f{W}{z}e^{\f{W}{z}}$). Note that $T(z) := -\f{W}{-z}$ is called the \emph{tree function} and appears 
in the enumeration of trees. 

Set $u = s e^{- t}$, and take the limit
\begin{equation*}
  \lim_{t\to\infty}  \fE[Big]{\exp \Bigl(\frac {sX_i(t)} {e^{t}}\Bigr)}
  = \lim_{t\to\infty} \phi_i\Bigl(t, \frac s {e^{t}}\Bigr) 
  = \frac {T(s)} s.
     \end{equation*}
The right-hand side in the latter equation is 
the moment generating function of a Lambert random variable $W^*$. Thus, in the limit
we have
 $$e^{-t}\begin{pmatrix} 
                   X_1(t)\\
                   X_2(t)\\
                   \vdots\\
                   X_c(t)
  \end{pmatrix} \convD \begin{pmatrix} 
                   W_1^*\\
                   W_2^*\\
                   \vdots\\
                   W_c^*
  \end{pmatrix},$$ 
and the components~$W_i^*$ of the limiting vector are independent
Lambert random variables, each of which is distributed like $W^*$.

\subsection{Ehrenfest-like random walk}
\label{sec:ehrenfest}

In this example we take the Ehrenfest navigation matrix
$$\matA = \begin{pmatrix} 
                   -\gamma&\gamma\\
                    \gamma&-\gamma
   \end{pmatrix}$$
for some positive $\gamma$. This Ehrenfest-like random walk is  in two dimensions.
Let us call the two coordinates of the walk $X(t)$ and $Y(t)$, i.e.,
in the previously used notation, $\vecX(t) = (X(t), Y(t))^\top$.
Note that the two vectors for the choice of movement are
in opposite directions and aligned along the $45$-degree line
\begin{equation}\label{invarianteq}
  X(t) + Y(t) = \lambda.
\end{equation}
for some intercept $\lambda$.
Thus, the movement is constrained to a linear subspace.  For
this walk to be tenable, both $X(0)/\gamma$ and $Y(0)/\gamma$ have to
be nonnegative integers (or alternatively, $\lambda/\gamma$ and
$X(0)/\lambda$ must be a positive integer).
With $u$ and $v$ being variables of the moment generating function,
the equation to solve is
$$\frac {\partial \phi(t; u, v)} {\partial t}  + (1 - e^{-\gamma u + \gamma v})
         \frac {\partial \phi(t; u, v)} {\partial u} + (1 - e^{\gamma u - \gamma v})
         \frac {\partial \phi(t; u, v)} {\partial v}  = 0.$$

Set $v= 0$ and $\eta(t, u) := \phi(t;u,0)$, and note that 
$\eta(t,u) = \fE[empty]{e^{uX(t)}}$ is the moment generating function of $X(t)$.
We can rewrite the partial differential equation as
\begin{equation*}
  \frac{\partial \eta(t,u)}{\partial t}  + (1 - e^{-\gamma u })
  \frac{\partial \eta(t,u)}{\partial u} + (1 - e^{\gamma u })
  \fE{Y(t)e^{uX(t)}}= 0.
\end{equation*}
Using the invariant in equation~(\ref{invarianteq}), we write the latter equation as
\begin{equation*}
\frac {\partial \eta(t,u)} {\partial t}  + (1 - e^{-\gamma u })
         \frac {\partial \eta(t,u)} {\partial u} + (1 - e^{\gamma u })
        \,  \fE{(\lambda - X(t))e^{uX(t)}}= 0.
      \end{equation*}
We thus have the simplified equation
\begin{equation*}
\frac {\partial \eta(t,u)} {\partial t}  + (e^{\gamma u} - e^{-\gamma u })
         \frac {\partial \eta(t,u)} {\partial u} +
         \lambda (1-e^{\gamma u}) \,\eta(t,u)= 0.
       \end{equation*}
This equation has the solution
\begin{equation*}
  \eta(t,u) = \Bigl( \frac{1 + e^{-2\gamma t+\gamma u}+ e^{\gamma u} - e^{-2\gamma t}}
  {1 - e^{-2\gamma t+\gamma u}+ e^{\gamma u} + e^{-2\gamma t}}\Bigr)^{X(0)/\gamma}
  \Bigl( \frac{1 - e^{-2\gamma t+\gamma u} +   e^{\gamma u} + e^{-2\gamma t}}
  {2}\Bigr)^{\lambda/\gamma}.
\end{equation*}
As $t\to\infty$ we have the limit
$$\lim_{t\to\infty }\eta(t,u) =    \Bigl( \frac{1 +  e^{\gamma u} }2\Bigr)^{\lambda/\gamma} . $$
Recall that the tenability requires that $\lambda/\gamma$ is a positive integer. Therefore,
the limit of the moment generating function is that of $\Bin(\lambda/\gamma, 1/2)$,
namely, a binomial random variable that counts the number of successes in
$\lambda/\gamma$ independent, 
identically distributed trials, with rate of success $1/2$ per trial.

Unlike the \polya--Eggenberger-like random walk,
the Ehrenfest-like random walk is not much affected by where it starts in
the first quadrant of the $X$$Y$-plane.
At any $t$,
the exact distribution does have $\vecX(0)$ in it; however, its influence
is attenuated exponentially fast in time, and in the limit it is completely
obliterated.
\subsection{Walking along a 45-degree hill}

For $\gamma\in \R^+$, the navigation matrix
\begin{equation*}
\matA = 
\begin{pmatrix}-\gamma  & -\gamma\\
               \gamma  & \gamma\end{pmatrix}
\end{equation*}
takes a walk along an oblique line, like climbing a 45-degree hill.
That walk will remain tenable so long as $Y(0) > X(0) $. 
Even if the navigator
walks all the way down to the bottom of the hill (hitting the $Y$-axis at a positive point),
the navigator will come back up along the 45-degree line with probability~1, staying in the first quadrant.
In this walk, we always add or subtract increments in the two dimensions
that are 
in the same amount; the difference $Y(t) - X(t) = Y(0) - X(0) = \lambda > 0$
remains the same at all times.  

As with the Ehrenfest-like random walk of the previous section,
specializing the partial differential equation of
Theorem~\ref{Theo:pde} to this walk along a $45$-degree hill, we again use
$v=0$ and $\eta(t, u) := \phi(t;u,0)$.  We note that 
\begin{equation*}
  \frac{\partial \eta(t,u)}{\partial t} 
  + (1 - e^{- \gamma u}) \frac{\partial \eta(t,u)}{\partial u}
  + (1 - e^{\gamma u})
  \Big(\frac{\partial \eta(t,u)}{\partial u} +\lambda  \eta(t,u)\Big) = 0.
\end{equation*}
Rearranging, we get
\begin{equation*}
  \frac{\partial \eta(t,u)} {\partial t}
  + (2 - e^{\gamma u} - e^{- \gamma u})\,
  \frac{\partial \eta(t,u)} {\partial u}
  + \lambda (1 - e^{\gamma u}) \eta(t,u)
  = 0.
\end{equation*}
This differential equation has the solution
\begin{align}\label{Eq:exactsol} 
  \eta(t,u) &= \frac {1}{(1 - \gamma t (e^{\gamma u}-1))^{\lambda/\gamma}}
  \Big( \frac{e^{\gamma u} - \gamma t (e^{\gamma u} - 1)}
  {1 - \gamma t (e^{\gamma u}-1)}\Big)^{X(0)/\gamma} \\
  &= \frac{(e^{\gamma u} - \gamma t (e^{\gamma u} - 1) )^{X(0)/\gamma}}
    {(1 - \gamma t (e^{\gamma u}-1))^{Y(0)/\gamma}} \notag
\end{align}

From the exact moment generating function we
get the exact mean and variance for the coordinates by taking derivatives at $u=0$; these quantities are
\begin{align*}
   \fE{X(t)} &= \lambda\gamma t+X(0),\\
   \fE{Y(t)} &= \lambda\gamma t+Y(0),\\
   \fV{X(t)} &= \fV{Y(t)}
   = \lambda \gamma^3 t^2 + (2 X(0) + \lambda)\gamma^2 t.
\end{align*} 
Next, we put $u = s/t$ in~(\ref{Eq:exactsol}) and use
the local expansion
\begin{equation*}
  e^{\gamma u} =
  e^{\gamma s/t} =
  1 + \frac {\gamma s} t +\Oh{\frac 1 {t^2}}.
\end{equation*}
We obtain
\begin{equation*}
  \fE{e^{s X(t) / t}}
  = \frac{1}{(1 - \gamma^2 s + \Oh{1/t})^{\lambda/\gamma}}
  \left(\frac{1 - \gamma^2 s + \Oh{1/t}}
    {1 - \gamma^2 s + \Oh{1/t}}\right)^{X(0)/\gamma}
\end{equation*}
which implies
\begin{equation*}
  \lim_{t\to\infty} \fE{e^{s X(t) / t}}
  = \frac{1}{(1 - \gamma^2 s)^{\lambda/\gamma}}.
\end{equation*}
The limiting moment generating function is that of a $\Gam(\lambda /
\gamma, \gamma^2)$ random variable, i.e., $X(t)/t$ converges in
distribution to  a $\Gam(\lambda / \gamma, \gamma^2)$random variable. We note that 
$Y(t)$ has a similar behavior.
} 

\section{Walks according to a balanced triangular scheme}
\label{sec:balanced-triangular}

In this section, we discuss a walk according to a balanced triangular
scheme.  The case has some historical significance. \polya\ urn models
came about in the first decades of the 20th century. 
Perhaps the  first urn schemes are Ehrenfest model,  constructed to
understand 
the diffusion of gas~\cite{Ehrenfest-Ehrenfest:1907:einwaende-boltzmann-H-theorem}, 
and  \polya-Eggenberger urn scheme, 
a model for contagion~\cite{Eggenberger-Polya:1923:statistik-verketteter-vorgaenge}.
 
Soon thereafter,
a theory was developed for many types of urns.  However, the
triangular flavor remained defiant until very recently.  The
triangular case has been handled in~\cite{Janson2006} and limit
distributions have been characterized. Alternative characterizations
are given in~\cite{PF188, Kuba-Mahmoud:2017:balanced-affine-urns, Zhang2015}.

So, let us consider a walk $\vecX(t) = (X(t), Y(t))^\top$
following a balanced triangular scheme with
the 
navigation matrix
\begin{equation*}
  \matA =
  \begin{pmatrix}
    \alpha & \delta-\alpha\\
    0 & \delta
  \end{pmatrix},
\end{equation*}
where $\alpha < \delta$ are numbers in~$\R^+$. We are
excluding the case $\alpha=\delta$; see Appendix~\ref{sec:polya-eggenberger},
where this \polya--Eggenberger like random walk has already been handled.

\begin{theorem}\label{thm:balanced-triangular}
  Suppose we have a balanced triangular scheme as described above.
  Then, the moment generating function is
  \begin{align*}
    \phi(t;u,v)
    &= \fE{e^{u X(t) + v Y(t)}} \\
    &=e^{-X(0)\, t} (
      e^{-\alpha u} - e^{-\alpha v} + 
      (e^{-\delta v} - 1 + e^{-\delta t})^{\alpha/\delta}
    )^{-X(0)/\alpha}\\
    &\phantom{=}\;
    \times e^{-Y(0)\, t} (e^{-\delta v} - 1 + e^{-\delta t})^{-Y(0)/\delta}.
\end{align*}
\end{theorem}

From this result we can compute the moments in a straightforward way.
See Section~\ref{sec:balanced-triangular:moments} for details.

To prove Theorem~\ref{thm:balanced-triangular}, we solve the
corresponding partial differential equation of Theorem~\ref{Theo:pde}
by using the method of characteristics. To improve readability, the
proof is split up into several sections.

\subsection{Characteristic curves}

We derive the characteristic curves belonging to the partial
differential equation by establishing the following lemma.

\begin{lemma}\label{lem:characteristic-curves}
  The functions
  \begin{equation*}
    x_c = e^{\alpha t} (e^{-\alpha u} - e^{-\alpha v} )
    \qquad\text{and}\qquad
    y_c = e^{\delta t} (e^{-\delta v} - 1)
  \end{equation*}
  are characteristic curves for the partial differential equation
  that corresponds to the balanced triangular scheme
  specified above.
\end{lemma}

To prove this lemma, consider the more general matrix
\begin{equation*}
  \matA =
  \begin{pmatrix}
    \alpha & \beta \\ \gamma & \delta
  \end{pmatrix}.
\end{equation*}

We want to find a solution $\phi(t; u, v)$ of the partial differential equation
\begin{equation*}
  \frac{\partial \phi}{\partial t}
  + (1-e^{\alpha u}e^{\beta v}) \frac{\partial \phi}{\partial u}
  + (1-e^{\gamma u}e^{\delta v}) \frac{\partial \phi}{\partial v}
  = 0.
\end{equation*}

\subsubsection{Parameterizing}

As a first step in the method of characteristics, we introduce a new
parameter~$s$. We set
\begin{subequations}\label{eq:ode}
  \begin{align}\setcounter{equation}{19}
    \label{eq:ode:t}\frac{\dd t}{\dd s} &= 1, \\
    \label{eq:ode:u}\frac{\dd u}{\dd s} &= 1-e^{\alpha u}e^{\beta v},
    \intertext{and}
    \label{eq:ode:v}\frac{\dd v}{\dd s} &= 1-e^{\gamma u}e^{\delta v}.
  \end{align}
\end{subequations}
By using the chain rule and inserting
\eqref{eq:ode}, we obtain
\begin{equation*}
  \frac{\dd}{\dd s} \phi(t(s); u(s), v(s)) = 0,
\end{equation*}
so our function is constant along the characteristics.

From~\eqref{eq:ode:t} it follows that $s = t + t_0$. We choose
$t_0=0$, thus $s=t$, which we use from now on.

\subsubsection{Specialization to the upper triangular case}

Since we are interested in the upper triangular case, we now 
specialize to $\beta=\delta-\alpha$ and $\gamma=0$. 
Thus, \eqref{eq:ode:v} becomes
\begin{equation*}
  \frac{\dd v}{\dd t} = 1-e^{\delta v},
\end{equation*}
and we can easily solve it by the separation of variables. This gives
\begin{equation*}
  v - \frac{1}{\delta} \log(1-e^{\delta v}) = t - y_0,
\end{equation*}
for some (constant) initial condition~$y_0$.
This is equivalent to
\begin{subequations}
\begin{align}
  \label{eq:exp-d-t-yp}
  e^{\delta (t-y_0)} &= \frac{e^{\delta v}}{1-e^{\delta v}} = \frac{1}{e^{-\delta v}-1},\\
  \label{eq:char-exp-d-yp}
  e^{\delta y_0} &= e^{\delta (t-v)} (1-e^{\delta v})
  = e^{\delta t}(e^{-\delta v} - 1),\\
  \intertext{and to}
  \label{eq:exp-d-v}
  e^{\delta v} &= \frac{1}{e^{-\delta (t-y_0)} + 1}
\end{align}
\end{subequations}
as well. In particular, $y_c=e^{\delta y_0}$
(together with Equation~\eqref{eq:char-exp-d-yp}) is the second of our two characteristic
curves of Lemma~\ref{lem:characteristic-curves}.

It remains to derive the first characteristic curve.
Inserting~\eqref{eq:exp-d-v} into~\eqref{eq:ode:u} yields the
differential equation
\begin{equation*}
    \frac{\dd u}{\dd t} =
    1 - \frac{e^{\alpha u}}{\left(e^{-\delta(t-y_0)} + 1\right)^{\beta/\delta}}.
\end{equation*}
With the help of a computational symbolic algebra system, we can 
solve this. We obtain
\begin{equation*}
  u(t) = t - \frac{1}{\alpha} \log\Bigl(
    \alpha x_0 + \frac{\alpha}{\alpha+\beta}
      e^{\alpha t} (e^{\delta (t-y_0)})^{\beta/\delta}
      H \Bigr),
\end{equation*}
where $H$ denotes the hypergeometric function
\begin{equation*}
  H = {_2F_1}\Bigl(\frac{\beta}{\delta},
        \frac{\alpha+\beta}{\delta};
        \frac{\alpha +\beta}{\delta}+1;
        -e^{\delta (t-y_0)} \Bigr),
\end{equation*}
and for some (constant) initial condition~$x_0$.

Solving for this $x_0$ gives
\begin{equation}\label{eq:char-xp}
  x_0 = \frac{1}{\alpha} e^{\alpha(t-u)}
  - \frac{1}{\alpha+\beta}
  e^{\alpha t} (e^{\delta (t-y_0)})^{\beta/\delta} H.
\end{equation}

\subsubsection{Balancing the triangular scheme}

At this point we consider the balanced triangular case and specialize to
$\beta=\delta-\alpha$. The characteristic curve~\eqref{eq:char-xp} is
now
\begin{equation*}
  x_0 = \frac{1}{\alpha} e^{\alpha(t-u)}
  - \frac{1}{\delta}
  e^{\alpha t} \left(-Z\right)^{1-\alpha/\delta} 
  {_2F_1}\big(1-\frac{\alpha}{\delta}, 1; 2; Z\big)
\end{equation*}
with $Z = - e^{\delta(t-y_0)} = 1/(1-e^{-\delta v})$,
cf.~\eqref{eq:exp-d-t-yp}.

We can simplify the hypergeometric function, namely,
\begin{equation*}
{_2F_1}\left(\mu, 1; 2; Z \right)
    = \frac{1 - (1-Z)^{1-\mu}}{(1-\mu)Z}.
  \end{equation*}
%
By using $1-\mu=\alpha/\delta$,
since $(Z-1)/Z=e^{-\delta v}$, we obtain
\begin{equation*}
  x_0 = \frac{1}{\alpha} e^{\alpha(t-u)}
  + \frac{1}{\alpha} e^{\alpha t}
  \frac{1-(1-Z)^{\alpha/\delta}}{(-Z)^{\alpha/\delta}}
  = \frac{e^{\alpha t}}{\alpha}
  (e^{-\alpha u} - e^{-\alpha v}
    + (e^{-\delta v}-1)^{\alpha/\delta}
  ).
\end{equation*}
Setting
\begin{equation*}
  x_c = \alpha x_0 - e^{\alpha y_0}
  = e^{\alpha t} (e^{-\alpha u} - e^{-\alpha v} )
\end{equation*}
(we also used \eqref{eq:char-exp-d-yp}) completes the proof of
Lemma~\ref{lem:characteristic-curves}.

\subsection{The general solution}

As we have both characteristic curves $x_c$ and $y_c$ now (cf.\
Lemma~\ref{lem:characteristic-curves}), we can write down the general
solution to our partial differential equation.  For some
function~$\widetilde\phi$, we have
\begin{equation}\label{eq:balanced-upper:general}
  \phi(t; u, v) = \widetilde\phi(
    e^{\alpha t} (e^{-\alpha u} - e^{-\alpha v} ),
    e^{\delta t} (e^{-\delta v} - 1)
  ).
\end{equation}
To determine $\widetilde\phi$, we need to take
the initial conditions into account.
\subsection{Initial conditions}

We have two possible cases to consider
for the initial conditions of the
solution~(\ref{eq:balanced-upper:general}). The case covered in this
section is $u=v$. Regardless of the type of transition we have, we are
adding $\delta$ balls to the urn, each time a transition takes place.
In the other case (see Section~\ref{sec:initial-t-0}), we set $u=0$, so
that we are only considering transitions with the second
direction of navigation.

Let $\tau(0) :=X(0)+Y(0)$.
The time
between the $\ell$th and $(\ell+1)$st transition is exponential, with
parameter $\tau(0)  + \ell\delta$, i.e., with expected time
$1/(\tau(0) +\ell\delta)$.  When the $(\ell+1)$st transition occurs (for any
$\tau(0) \geq 0$), we walk a total of $\delta$ units in the sense of Manhattan taxicab geometry.  So,
the walk is $\tau(0) +\ell\delta$ Manhattan blocks away from the origin,
after $\ell$
transitions have taken place.

\begin{lemma}\label{lem:init-pgf-u-is-v}
  If $u=v$, then the probability generating function of $Y(t)$ is
  \begin{equation*}
    \fE{e^{v Y(t)}}
    = (1 - e^{\delta t} + e^{\delta(t-v)}
    )^{-(X(0)+Y(0))/\delta}.
  \end{equation*}
\end{lemma}

Note that
\begin{equation*}
  \fE{e^{v Y(t)}} = \left(1 + y_c\right)^{-(X(0)+Y(0))/\delta},
\end{equation*}
where $y_c$ is the characteristic curve of
Lemma~\ref{lem:characteristic-curves}.

Lemma~\ref{lem:init-pgf-u-is-v} can be proved by setting up and
solving Kolmogorov's Forward Equations; see
Appendix~\ref{sec:details-proof-balanced} for details.

\def\movekolmogorov{
As mentioned in Section~\ref{sec:balanced-triangular}, we will setup
and solve Kolmogorov's Forward Equations to prove
Lemma~\ref{lem:init-pgf-u-is-v}.

\subsubsection{Kolmogorov's forward equations}

Let $P_{i,j}(t)$ denote the probability that, starting with $i$ balls
in the urn at a certain time, then $t$ time units later, we have $j$
balls in the urn.
In particular, $P_{i,i+\ell\delta}(t)$ is the probability that,
starting with $i$ balls, we have
exactly $\ell$ transitions during the next $t$ time units.
We follow some of the notation of Ross~\cite{Ross:1996:stochastic-proc}.
We let
$v_{i}$ denote the rate for the exponential distribution of time until
the next transition occurs, when there are currently $i$ balls in the
urn.  In our case, since the balls act independently, and each ball
has exponential rate 1 of being chosen, we have $v_{i} = i$.
We define $q_{i,j} = v_{i}P_{i,j}$, so $q_{i,i+\delta} = v_{i} = i$,
and $q_{i,j} = 0$ otherwise.

We have, as in Ross's Lemma 5.4.1,
\begin{equation*}
\lim_{t\rightarrow 0}\frac{1-P_{i,i}(t)}{t} = v_{i}
\qquad\text{and}\qquad
\lim_{t\rightarrow 0}\frac{P_{i,j}(t)}{t} = q_{i,j},\qquad\text{for $i\neq j$}.
\end{equation*}
So, now we
set up Kolmogorov's Forward Equations, following
Theorem~5.4.4 of Ross.
In our case, these equations are
\begin{subequations}
  \label{eq:kolmogorov-forward}
  \begin{equation}
    P_{i,i}^{\prime}(t) = - v_{i}P_{i,i}(t) = - i P_{i,i}(t),
  \end{equation}
  and for $\ell \geq 1$, 
  \begin{equation}
    P_{i,i+\ell\delta}^{\prime}(t) = (i+(\ell-1)\delta)P_{i,i+(\ell-1)\delta}(t)
    - (i+\ell\delta)P_{i,i+\ell\delta}(t).
  \end{equation}
\end{subequations}

We use the rising factorial notation
$(i/\delta)^{\overline{\ell}} := \prod_{k=0}^{\ell-1}(i/\delta+k)$
in the statement of the lemma.
The initial conditions on the $P_{i,i+\ell\delta}(t)$ are
$P_{i,i}(0)=1$ and $P_{i,i+\ell\delta}(0)=0$ for $\ell\geq1$.

\begin{lemma}\label{lem:solution-kolmogorov}
  The functions
  \begin{equation*}
    P_{i,i+\ell\delta}(t) =
    \frac{(i/\delta)^{\overline{\ell}}}{\ell!}
    e^{-it}(1 - e^{-\delta t})^{\ell}
  \end{equation*}
  are the solutions to the Kolmogorov system of differential
  equations~\eqref{eq:kolmogorov-forward}.
\end{lemma}

Before proving Lemma~\ref{lem:solution-kolmogorov}, we first note this indeed is a probability
distribution, since $\sum_{\ell \geq 0} P_{i,i+\ell\delta}(t) = 1$ (follows from
Lemma~\ref{lem:init-pgf-u-is-v}) and since all these summands are
nonnegative.

In particular, we have the solutions
\begin{equation*}
  P_{i,i}(t) = e^{-it}
\qquad\text{and}\qquad
  P_{i,i+\delta}(t) = \frac{i}{\delta}e^{-it}(1-e^{-\delta t}).
\end{equation*}
\begin{proof}[Proof of Lemma~\ref{lem:solution-kolmogorov}]
  Inserting $t=0$ shows that the initial conditions are satisfied.  We
  used the conventions that $0^0=1$, and that the empty product equals $1$.
  The case $\ell=0$ follows by a direct calculation.

  When $\ell\geq1$ we have
  \begin{align*}
    P_{i,i+\ell\delta}^{\prime}(t)
    &= \frac{\prod_{k=0}^{\ell-1}(i+k\delta)}{\ell!\thinspace\delta^{\ell}}
    (- i e^{-it}(1 - e^{-\delta t})^{\ell}
      +\ell\delta e^{-it} e^{-\delta t} (1 - e^{-\delta t})^{\ell-1}
    ) \\
    &=- i
    \frac{\prod_{k=0}^{\ell-1}(i+k\delta)}{\ell!\thinspace\delta^{\ell}}
    e^{-it}(1 - e^{-\delta t})^{\ell} \\
    &\phantom{=}\hphantom{0}- \ell\delta 
    \frac{\prod_{k=0}^{\ell-1}(i+k\delta)}{\ell!\thinspace\delta^{\ell}}
    e^{-it} (1 - e^{-\delta t}) (1 - e^{-\delta t})^{\ell-1} \\
    &\phantom{=}\hphantom{0}+
    (i+(\ell-1)\delta)
    \frac{\prod_{k=0}^{\ell-1}(i+k\delta)}{(\ell-1)!\thinspace\delta^{\ell-1}}
    e^{-it} e^{-\delta t} (1 - e^{-\delta t})^{\ell-1} \\
    &= - (i+\ell\delta)P_{i,i+\ell\delta}(t)
    + (i+(\ell-1)\delta)P_{i,i+(\ell-1)\delta}(t),
  \end{align*}
  which proves the lemma.
\end{proof}

\subsubsection{Probability generating function}

By using the solutions to Kolmogorov's Forward Equations (Lemma~\ref{lem:solution-kolmogorov}), it is not hard anymore to derive the moment generating function~$\fE[empty]{e^{v Y(t)}}$.

\begin{proof}[Proof of Lemma~\ref{lem:init-pgf-u-is-v}]
  To calculate $\fE[empty]{e^{v Y(t)}}$, we insert $i=X(0)+Y(0)$ at the end of this
  proof; but for the moment we still write the~$i$.
  Taking Lemma~\ref{lem:solution-kolmogorov}
  and summing yields
  \begin{equation*}
    \fE{e^{v Y(t)}} = \sum_{\ell \geq 0} P_{i,i+\ell\delta}(t) e^{u(i+\delta \ell)}
    = e^{i(v-t)} \sum_{\ell \geq 0} \frac{(i/\delta)^{\overline{\ell}}}{\ell!}
    (e^{\delta v}(1 - e^{-\delta t}))^{\ell}.
  \end{equation*}
  Since $(1-Z)^{-\mu} = \sum_{\ell\geq0} \mu^{\overline{\ell}} Z^\ell\! /
  \ell!$ (again using the rising factorial notation
  $\mu^{\overline{\ell}}:=\mu(\mu+1)\dots(\mu+\ell-1)$), we obtain
  \begin{equation*}
    \fE{e^{v Y(t)}}
    = e^{-i(t-v)} (1 - e^{\delta v}(1 - e^{-\delta t})
    )^{-i/\delta},
  \end{equation*}
  and the result follows by rearranging the terms.
\end{proof}

\subsection{Second initial condition}
\label{sec:initial-t-0}

For this initial condition, we consider the situation at time
$t=0$. Again, we have $X(0)$ units of the first type and $Y(0)$
units of the second type. This translates to the probability
generating function
\begin{align*}
  \fE{e^{u X(0) + v Y(0)}} 
  &= e^{u X(0)} e^{v Y(0)} \\
  &= \left(x_c\vert_{t=0} + 
    \left(y_c\vert_{t=0} + 1\right)^{\alpha/\delta}
  \right)^{-X(0)/\alpha}
  \left(y_c\vert_{t=0} + 1\right)^{-Y(0)/\delta},
\end{align*}
which we have by rewriting in terms of the characteristic curves $x_c$
and $y_c$ of Lemma~\ref{lem:characteristic-curves}.
} 

\subsection{Solution to the partial differential equation}

We are now ready to determine the function~$\widetilde\phi$ of the
general solution~\eqref{eq:balanced-upper:general}, and thus, we will prove
Theorem~\ref{thm:balanced-triangular}.

\begin{proof}[Proof of Theorem~\ref{thm:balanced-triangular}]
  We use the initial conditions provided by
  Lemma~\ref{lem:init-pgf-u-is-v}.
We compute
\begin{align*}
  \phi(t; u,v)
  &= \fE{e^{u X(t) + v Y(t)}} \\
  &= (x_c + (y_c + 1)^{\alpha/\delta})^{-X(0)/\alpha}
  (y_c + 1)^{-Y(0)/\delta} \\
  &= (e^{\alpha t}
    (e^{-\alpha u} - e^{-\alpha v} ) + 
    (e^{\delta t} (e^{-\delta v} - 1) + 1)^{\alpha/\delta}
  )^{-X(0)/\alpha} \\
  & \qquad\qquad \times (
    e^{\delta t} (e^{-\delta v} - 1) + 1
  )^{-Y(0)/\delta} \\
  &=e^{-X(0) t} (
    e^{-\alpha u} - e^{-\alpha v} + 
    (e^{-\delta v} - 1 + e^{-\delta t})^{\alpha/\delta}
  )^{-X(0)/\alpha} \\
  &\qquad\qquad \times e^{-Y(0) t} (e^{-\delta v} - 1 + e^{-\delta t})^{-Y(0)/\delta}.
\end{align*}
This solution 
satisfies our initial conditions, and our proof is complete.
\end{proof}
\subsection{Moments}
\label{sec:balanced-triangular:moments}

Theorem~\ref{thm:balanced-triangular} provides us with the moment
generating function. The moments now follow by differentiation and
evaluating at $u=0$ and $v=0$. The first moments are
\begin{align*}
  \fE{X(t)} &=
  X(0) e^{\alpha t},\\
  \fE{Y(t)} &=
  (X(0)+Y(0)) e^{\delta t}
  - X(0) e^{\alpha t}.
\end{align*}
Note that the two coordinates of the walk grow at different rates, 
and the drift is much stronger in the vertical direction. 

We can also compute the second moments,
\begin{align*}
  \fE[empty]{X^2(t)} &=
  X(0)(\alpha+X(0)) e^{2\alpha t}
  - \alpha X(0) e^{\alpha t},\\
  \fE{X(t)Y(t)} &=
  X(0)(\alpha+X(0)+Y(0)) e^{(\alpha+\delta) t}
  - X(0)(\alpha+X(0)) e^{2\alpha t},\\
  \fE[empty]{Y^2(t)} &=
  (X(0)+Y(0))(\delta+X(0)+Y(0)) e^{2\delta t}\\
  &\phantom{=}\hphantom{0}{}
  - 2 X(0)(\alpha+X(0)+Y(0)) e^{(\alpha+\delta) t}\\
  &\phantom{=}\hphantom{0}{}
  {}+ \delta(X(0)+Y(0)) e^{\delta t}
  {}+ X(0)(\alpha+X(0)) e^{2\alpha t}
  {}+ \alpha X(0) e^{\alpha t},
  \intertext{and consequently}
  \fV{X(t)} &=
  \alpha X(0) e^{2\alpha t}
  - \alpha X(0) e^{\alpha t},\\
  \fCov{X(t),Y(t)} &=
  \alpha X(0) e^{(\alpha+\delta) t}
  - \alpha X(0) e^{2\alpha t},\\
  \fV{Y(t)} &=
  \delta (X(0)+Y(0)) e^{2\delta t}
  - 2 \alpha X(0) e^{(\alpha+\delta) t}\\
  &\phantom{=}\hphantom{0}{}
  - \delta(X(0)+Y(0)) e^{\delta t}
  + \alpha X(0) e^{2\alpha t}
  + \alpha X(0) e^{\alpha t}.
\end{align*}

\begin{remark}\label{remark:first-moment}
  Of course, the first moments follow by
  Theorem~\ref{Theo:polyaprocave} as well: We have
  \begin{equation*}
    \fE{\vecX(t)} = e^{
      {\scriptsize\begin{pmatrix}
          \alpha  & 0\\ \delta-\alpha  & \delta
        \end{pmatrix}} t}\,  \vecX (0).
  \end{equation*}
  The matrix in the exponent is diagonal, which makes it an easy
  computation and we obtain
  \begin{equation*}
    \begin{pmatrix} \fE{X(t)} \\ \fE{Y(t)} \end{pmatrix}
    = 
    \begin{pmatrix}
      e^{\alpha t}  & 0\\
      e^{\delta t}-e^{\alpha t}  & e^{\delta t}
    \end{pmatrix}
    \begin{pmatrix} X(0) \\ Y(0) \end{pmatrix},
  \end{equation*}
as derived by differentiation.  
\end{remark}
\subsection{Asymptotic distributions}
With the exact moment generating function at hand, we can determine marginal 
as well as joint distributions of the number of (suitably scaled) white and blue balls. 

Let us start with the displacement in the first coordinate. Set $v=0$, 
and evaluate~$X(t)$ with the scale
$e^{\alpha t}$. We see that
\begin{equation*}
\E\Bigl[\exp\Bigl(\frac {X(t)} {e^{\alpha t}} u\Bigr)\Bigr] 
      = \phi \Bigl(t; \frac u {e^{\alpha t}}, 0 \Bigr) = 
         \frac {e^{-\tau t} e^{bt}} {(e^{-\alpha u e^{-\alpha t}} -1 + e^{-\alpha t})
             ^{X(0) / \alpha}}.  
\end{equation*}
As $t \to \infty$, we find
\begin{align*}
\E\Bigl[\exp\Bigl(\frac {X(t)} {e^{\alpha t}} u\Bigr)\Bigr] 
      &= 
         \frac {e^{-\tau t} e^{bt}} {((1 - \alpha u e^{-\alpha t} + O (e^{-2 \alpha t})) -1 + e^{-\alpha t})
             ^{X(0) / \alpha}} \\
             &\sim \frac {e^{-\tau t} e^{Y(0)t}} {(e^{-\alpha t}- \alpha u e^{-\alpha t} )
             ^{X(0) / \alpha}}  \\
             &\sim \frac {e^{-\tau t} e^{Y(0)t}} {e^{-X(0)t}(1- \alpha u)
             ^{X(0) / \alpha}},
\end{align*} 
where ``$\sim$'' stands for asymptotic equivalence.
  Recalling that $\tau(0) = X(0) + Y(0)$, we see a cancellation leading
  to the convergence
  \begin{equation*}
    \E\Bigl[\exp\Bigl(\frac {X(t)} {e^{\alpha t}} u\Bigr)\Bigr] 
     \to \frac 1 {(1- \alpha u)
             ^{X(0) / \alpha}}.
  \end{equation*}
  The right-hand side is the moment generating function of a $\Gam(X(0)/\alpha, \alpha)$ random variable.
  By L\'evy's Continuity Theorem~\cite{Williams:1991:probability-martingales} (Theorem 18.1), we have
  \begin{equation*}
    \frac {X(t)} {e^{\alpha t}} \to \Gam\Bigl(\frac {X(0)}
    \alpha ,\alpha\Bigr).
  \end{equation*}
  By a similar analysis, putting $u= 0$ and using the scale $e^{\delta t}$ for
  the displacement along the second axis, we find 
  \begin{equation*}
    \frac {Y(t)} {e^{\alpha t}} \to \Gam\Bigl(\frac {\tau(0)}
    \delta,\delta \Bigr).
  \end{equation*}
Joint distributions with the two marginals being gamma distributions are called 
bivariate gamma distributions. From the variances and covariance computed in 
Subsection~\ref{sec:balanced-triangular:moments}, we find the asymptotic correlation between $X(t) e^{-\alpha t}$ and  $Y(t) e^{-\delta t}$ to be
$\sqrt { (\alpha X(0))/(\delta \tau(0))}$, for $\alpha < \delta$.
This extends a calculation in~\cite{ Chen-Mahmoud:2016:time-continuous-polya}, under the scenario of 
a triangular \polya\ urn scheme.
For the case $\alpha = \delta$, we have 0~correlation, as naturally arises
from the independence of the movement along the two axes. 

\def\movesecondmoment{
  Pursuing a similar approach to the one we used to derive the mean in
  Theorem~\ref{Theo:polyaprocave}, we can try to go forward with the
  second moment. We only highlight the salient steps. Take the
  partial derivatives $\partial^2\! / \partial u^2$, $\partial^2\!
  / \partial u\!\ \partial v$, and $\partial^2\! / \partial v^2$ of the
  moment generating function of Theorem~\ref{thm:balanced-triangular},
  and evaluate each equation at $u=0$ and $v=0$. We obtain the system of
  ordinary differential equations
  \begin{align*}
    \frac d {dt} \fE[empty]{X^2(t)}
    &= 2\alpha \fE[empty]{X^2(t)} + \alpha^2\fE{X(t)}, \\
    \frac d {dt} \fE{X(t)Y(t)}
    &= (\alpha+\delta)\fE{X(t)Y(t)} + \beta\fE[empty]{X^2(t)}
    + \alpha\beta \fE{X(t)} , \\
    \frac d {dt} \fE[empty]{Y^2(t)}
    &=  2\delta\fE[empty]{Y^2(t)} + 2\beta \fE{X(t)Y(t)}
    + \beta^2\fE[empty]{X(t)} + \delta^2 \fE{Y(t)},
  \end{align*}
  with $\beta=\delta-\alpha$.
  This system is to be solved under the initial conditions $\fE[empty]{X^2(0)}
  = X(0)^2$, $\fE{X(0)Y(0)} = X(0)Y(0)$, and $\fE[empty]{Y^2(0)} = Y(0)^2$.

  We can solve it sequentially, starting with $\fE{X^2(t)}$, as its
  differential equation is self contained. We can then plug in the
  solution of $\fE[empty]{X^2(t)}$ into the differential equation for
  $\fE{X(t)Y(t)}$, which at this point would have only known components
  on the right-hand side.  Finally, we plug in all the known functions
  of averages and mixed moments in the differential equation for
  $\fE[empty]{Y^2(t)}$. 
} 

\newcommand{\MR}[1]{}
{\footnotesize
\bibliographystyle{amsplain}
\bibliography{continuumPolya}}

\vfill
\bigskip
{\footnotesize
\noindent\begin{minipage}{1.0\linewidth}
\noindent
Daniel Krenn\\
Department of Mathematics\\
Alpen-Adria-Universität Klagenfurt\\
Universitätsstraße 65--67\\
9020 Klagenfurt am Wörthersee, Austria\\
\href{mailto:math@danielkrenn.at}{\tt math@danielkrenn.at} or
\href{mailto:daniel.krenn@aau.at}{\tt daniel.krenn@aau.at}\\
\end{minipage}

\noindent\begin{minipage}{1.0\linewidth}
\noindent
Hosam~M.\@ Mahmoud\\
Department of Statistics\\
The George Washington University\\
Washington, D.C.~20052, USA\\
\href{mailto:hosam@gwu.edu}{\tt hosam@gwu.edu}\\
\end{minipage}

\noindent\begin{minipage}{1.0\linewidth}
\noindent
Mark Daniel Ward\\
Department of Statistics\\
Purdue University\\
West Lafayette, IN 47907, USA\\
\href{mailto:mdw@purdue.edu}{\tt mdw@purdue.edu}\\
\end{minipage}

}

\clearpage
\appendix

\section{Illustrative examples}
\label{sec:examples}
We give here some examples. Some of them resemble and extend
standard \polya\ processes to the random walk counterpart. Some have
no solved equivalent in the \polya\ world (neither the discrete- or continuous-time
versions).

\moveexamples

\section{Left-out details of Section~\ref{sec:balanced-triangular}}
\label{sec:details-proof-balanced}
\subsection{First initial condition}

\movekolmogorov

\section{Alternative approach for the second moment of balanced
  triangular schemes}
\label{sec:alt-second-moment}

In Remark~\ref{remark:first-moment} we state an alternative approach
for obtaining the first moment of the balanced triangular scheme of
Section~\ref{sec:balanced-triangular}.

\movesecondmoment

\end{document}